\documentclass[12pt,a4paper]{article}

\usepackage{amsmath,amsfonts,amssymb}
\usepackage{bbm}
\usepackage{cases}

\usepackage{lipsum}
\usepackage{subcaption}

\usepackage{xcolor}
\usepackage{graphicx}
\usepackage{ulem}

\newtheorem{defn}{Definition}[section]
\newtheorem{theorem}[defn]{Theorem}

\newtheorem{proposition}[defn]{Proposition}
\newtheorem{corollary}[defn]{Corollary}
\newtheorem{remark}[defn]{Remark}

\newenvironment{proof}{{\bf Proof }}{{\vskip 0.1cm \hfill$\Box$}}

\begin{document} 

\noindent
{\Large \bf Remarks on well-posedness for linear elliptic equations via divergence-free transformation}
\\ \\
\bigskip
\noindent
{\bf Haesung Lee}  \\
\noindent
{\bf Abstract.}  
This paper investigates the well-posedness of linear elliptic equations, focusing on the divergence-free transformation introduced in the author's recent work [J. Math. Anal. Appl. 548 (2025), 129425]. By comparing this approach with classical bilinear form methods, we demonstrate that while standard techniques encounter limitations in handling zero-order coefficients $c \in L^1(U)$, the divergence-free transformation successfully establishes well-posedness in this setting. Furthermore, utilizing the Riesz-Thorin interpolation theorem between the cases $c \in L^1(U)$ and $c \in L^{\frac{2d}{d+2}}(U)$, we establish the existence and uniqueness of weak solutions under the assumption $c \in L^s(U)$ for $s \in [1, \frac{2d}{d+2}]$.

\section{Introduction}
This paper provides an overview of recent results concerning the well-posedness of linear elliptic equations, focusing on methods that utilize a divergence-free transformation developed in \cite{L25jm}. Linear elliptic partial differential equations, in both divergence and non-divergence forms, are foundational tools for modeling phenomena across physics, financial engineering, and pure mathematics. A central theme in their study has been the existence, uniqueness, and regularity of solutions to the following Dirichlet boundary value problem:
\begin{equation} \label{maineq2}
\left\{
\begin{alignedat}{2}
-{\rm div}(A \nabla u) + \langle \mathbf{H}, \nabla u \rangle + c u &=f  && \quad \; \mbox{\rm in $U$,}\\
u &= 0 &&\quad \; \mbox{\rm on $\partial U$},
\end{alignedat} \right.
\end{equation}
where $U \subset \mathbb{R}^d$ is a bounded open set with $d \geq 3$. The matrix $A$ is assumed to be uniformly strictly elliptic (i.e. \eqref{elliptici} holds), and the precise regularity conditions on the coefficients and data will be discussed below. Here, $u$ is called a weak solution to \eqref{maineq2} if $u \in H^{1,2}_0(U)$ and satisfies
\begin{equation} \label{weakformu}
\int_U \left( \langle A \nabla u, \nabla \psi \rangle + \langle \mathbf{H}, \nabla u \rangle \psi + cu \psi \right) \, dx = \int_U f \psi \, dx \quad \text{ for all } \psi \in C_0^{\infty}(U),
\end{equation}
where the integrals are well defined (for instance, $\mathbf{H} \in L^2(U, \mathbb{R}^d)$, $c \in L^1(U)$ with $cu \in L^1(U)$, and $f \in L^1(U)$). Important tools for studying the well-posedness of \eqref{maineq2} are the Lax-Milgram theorem, the Fredholm alternative, and the weak maximum principle. These methods were systematically employed in \cite{S65} to establish existence and uniqueness of weak solutions, and the resulting framework was later generalized in \cite{T73} (see also \cite{Tma}). However, the classical approaches of \cite{S65} and \cite{T73} typically require assumptions such as
$\mathbf{H} \in L^d(U, \mathbb{R}^d)$, $c \in L^{\frac{d}{2}}(U)$,
in order to guarantee that the associated bilinear form is bounded on $H^{1,2}_0(U) \times H^{1,2}_0(U)$.
In particular, the integrability of the zero-order coefficient $c$ seems difficult to generalize in any natural way beyond $L^{\frac{d}{2}}(U)$ (cf. \cite{Dr} and \cite{GGM13}), and this appears to be the case for $W^{1,p}$-weak solutions as well (see \cite[Corollary 2.1.6]{BKRS15} and \cite{KK19, KPT23}). Furthermore, in \cite{V92, K21}, even for establishing the existence of strong solutions to the non-divergence form counterpart of \eqref{maineq2}, the condition $c \in L^{\frac{d}{2}}(U)$ is required for the zero-order coefficient.
However, by employing a new method based on the so-called divergence-free transformation, developed in \cite{L25jm}, the unique weak solution to \eqref{maineq2} can be transformed into the unique weak solution to \eqref{maineqdivfree}. Consequently, we can also handle the case where the zero-order coefficient $c$ belongs merely to $L^1(U)$, which substantially relaxes the previous integrability assumption $c \in L^{\frac{d}{2}}(U)$.
\\
In this paper, we discuss the existence, uniqueness, energy estimates, and boundedness estimates for solutions to \eqref{maineq2} obtained in \cite{L25jm}, and additionally investigate the trade-off between the integrability of the data $f$ and that of the coefficient $c$ via the Riesz-Thorin interpolation theorem. Ultimately, we establish the well-posedness of \eqref{maineq2} even when the zero-order coefficient $c$ satisfies only $c \in L^s(U)$ with $s \in [1, \frac{2d}{d+2}]$. Here the corresponding data $f$ is assumed to satisfy an intermediate integrability condition between $L^{\frac{\hat{p}d}{d+\hat{p}}}(U)$ and $L^{\frac{2d}{d+2}}(U)$ for some $\hat{p} \in (d, \infty)$.
\\
In particular, this paper should be an overview of \cite{L25jm} rather than a work providing complete proofs. We focus on stating the main results and explaining the underlying ideas. The paper is organized as follows. In Section \ref{secnot}, we introduce notations and conventions. In Section \ref{sec3} we discuss existence and uniqueness of solutions via the classical bilinear form method. Section \ref{sec4} introduces divergence-free transformation methods, and Section \ref{sec5} is devoted to deriving the interpolation results.

\section{Notations and conventions} \label{secnot}
In this paper, we work on the Euclidean space $\mathbb{R}^d$ equipped with the standard inner product $\langle \cdot, \cdot \rangle$ and norm $\|\cdot\|$.
For $a,b \in \mathbb{R}$ we set $a \wedge b := \min\{a,b\}$ and $a \vee b := \max\{a,b\}$. 
For a Lebesgue measurable set $E \subset \mathbb{R}^d$ we write $|E| := dx(E)$, where $dx$ is the Lebesgue measure on $\mathbb{R}^d$.
Let $U \subset \mathbb{R}^d$ be open. We denote by $\mathcal{B}(U)$ the set of Borel measurable functions on $U$. If $\mathcal{A} \subset \mathcal{B}(U)$, then $\mathcal{A}_0$ is the set of all $f \in \mathcal{A}$ whose support is compactly contained in $U$. The spaces of continuous functions on $U$ and on $\overline{U}$ are denoted by $C(U)$ and $C(\overline{U})$, and we put $C_0(U) := C(U)_0$. For $k \in \mathbb{N} \cup \{\infty\}$ we denote by $C^k(U)$ the space of $k$-times continuously differentiable functions on $U$ and set $C_0^k(U) := C^k(U) \cap C_0(U)$. 
For $r \in [1,\infty]$ we denote by $L^r(U)$ the usual Lebesgue space on $U$ with respect to $dx$, equipped with the norm $\|\cdot\|_{L^r(U)}$. The space $L^r(U,\mathbb{R}^d)$ consists of all $\mathbb{R}^d$-valued $L^r$-functions on $U$, with norm $\|\mathbf{F}\|_{L^r(U)} := \big\|\,\|\mathbf{F}\|\,\big\|_{L^r(U)}$. 
For $i = 1,\dots,d$ and a function $f$ on $U$ we denote by $\partial_i f$ the $i$-th weak partial derivative of $f$, whenever it exists. The Sobolev space $H^{1,r}(U)$ consists of all $f \in L^r(U)$ such that $\partial_i f \in L^r(U)$ for all $i = 1,\dots,d$, endowed with the usual $H^{1,r}(U)$-norm. For $q \in [1,\infty)$ we write $H_0^{1,q}(U)$ for the closure of $C_0^\infty(U)$ in $H^{1,q}(U)$. 
In the case where $U$ is a bounded open subset of $\mathbb{R}^d$, by the Poincaré inequality, for each $u \in H^{1,2}_0(U)$ we write
$ \|u\|_{H^{1,2}_0(U)} := \|\nabla u\|_{L^2(U)}$.
We denote the dual space of $H^{1,2}_0(U)$ by $H^{-1,2}(U)$. 
The space $H^{2,r}(U)$ is defined as the set of all $f \in L^r(U)$ such that $\partial_i f \in L^r(U)$ and $\partial_i \partial_j f \in L^r(U)$ for all $i,j = 1,\dots,d$, equipped with the usual $H^{2,r}(U)$-norm. The weak Laplacian is given by $\Delta f := \sum_{i=1}^d \partial_i \partial_i f$.

\section{Bilinear form methods} \label{sec3}
Let us consider the following condition: \\ \\
{\bf (Y1)}:
\textit{
$U$ is a bounded open subset of $\mathbb{R}^d$ with $d \geq 3$, $\mathbf{H} \in L^d(U, \mathbb{R}^d)$, $c \in L^{\frac{d}{2}}(U)$ with $c \geq 0$ in $U$. $A = (a_{ij})_{1 \leq i,j \leq d}$ is a (possibly non-symmetric) matrix of measurable functions on $\mathbb{R}^d$ such that $A$ is uniformly strictly elliptic and bounded, i.e.
there exist constants $M > 0$ and $\lambda > 0$ satisfying
\begin{equation} \label{elliptici}
\langle A(x) \xi, \xi \rangle \geq \lambda \| \xi \|^2, \;\; \max_{1 \leq i,j \leq d} |a_{ij}(x)| \leq M, 
\quad \text{for a.e. } x \in \mathbb{R}^d \text{ and all } \xi \in \mathbb{R}^d.
\end{equation}
}
Under the assumption \textbf{(Y1)}, we define the bilinear form $(\mathcal{B}, H^{1,2}_0(U))$
by
\begin{equation} \label{bilinearf}
\mathcal{B}(f,g):= \int_{U} \langle A \nabla f, \nabla g \rangle \,dx + \int_{U} \langle \mathbf{H}, \nabla f \rangle g\,dx + \int_{U} c f g\,dx, \quad f,g \in H^{1,2}_0(U).
\end{equation}
Then, it is well-defined and, by the H\"older inequality and the Sobolev inequality \cite[Section 5.6, Theorem 1]{E10}, it satisfies
\begin{equation} \label{boundedness}
\left|  \mathcal{B} (f,g) \right| \leq K \| \nabla f\|_{L^2(U)} \| \nabla g\|_{L^2(U)},
\end{equation}
where
\[
K := dM + \frac{2(d-1)}{d-2} \|\mathbf{H} \|_{L^d(U)} + \frac{4(d-1)^2}{(d-2)^2} \|c\|_{L^{\frac{d}{2}}(U)}.
\]
Since $\mathbf{H}\in L^{d}(U, \mathbb{R}^{d})$, we have
\[
\lim_{n \to\infty}\int_{U} 1_{\{\|\mathbf{H}\|\ge n \}}\|\mathbf{H}\|^{d}\,dx=0.
\]
Hence, we can choose a constant $N\ge0$ such that
\begin{equation} \label{vectorfieh}
\left( \int_U 1_{\{ \|\mathbf{H}\| \geq N \}} \| \mathbf{H} \|^d \, dx \right)^{\frac{2}{d}} 
\leq \frac{\lambda^2}{16} \left( \frac{d-2}{d-1} \right)^2.
\end{equation}
Then, we obtain

\begin{equation} \label{coercive}
\begin{aligned}[b]
&\left| \int_{U} \langle \mathbf{H}, \nabla f \rangle f \, dx \right|
\leq \int_U \|\mathbf{H}\|\, |f|\, \|\nabla f\| \, dx \\
&\quad \leq \frac{\lambda}{4} \int_U \|\nabla f\|^2 dx + \frac{1}{\lambda} \int_U \|\mathbf{H}\|^2 |f|^2 dx \\
& \quad =  \frac{\lambda}{4} \int_U \|\nabla f\|^2 dx +\frac{1}{\lambda}\int_U 1_{\{ \|\mathbf{H}\| \geq N \}} \|\mathbf{H}\|^2 |f|^2 dx \\
&\qquad +\frac{1}{\lambda} \int_U 1_{\{ \|\mathbf{H}\| < N \}} \|\mathbf{H}\|^2 |f|^2 dx \\
&\quad \leq \frac{\lambda}{4} \int_U \|\nabla f\|^2 dx+ \frac{1}{\lambda}\left( \int_U 1_{\{ \|\mathbf{H}\| \geq N \}} \|\mathbf{H}\|^d dx \right)^{\frac{2}{d}} \|f\|_{L^{\frac{2d}{d-2}}(U)}^2 \\
&\qquad + \frac{N^2}{\lambda} \int_U |f|^2 dx \\
&\quad \leq \frac{\lambda}{2} \|\nabla f\|_{L^2(U)}^2 + \frac{N^2}{\lambda} \int_U |f|^2 dx.
\end{aligned}
\end{equation}
Then,
\begin{equation} \label{desiredestim2}
\mathcal{B}(f,f) +\frac{N^2}{\lambda} \| f\|^2_{L^2(U)}  \geq \frac{\lambda}{2}\|\nabla f \|^2_{L^2(U)} \quad \text{ for all $f \in H^{1,2}_0(U)$}.
\end{equation}
Let $\gamma:= \frac{N^2}{\lambda}$. Since \eqref{boundedness} and \eqref{coercive} hold, it follows from the Lax-Milgram theorem that for each $\psi \in H^{-1,2}(U)$ there exists a unique $u_{\gamma, \psi} \in H^{1,2}_0(U)$ satisfying
$$
\mathcal{B}(u_{\gamma, \psi}, \varphi) + \frac{N^2}{\lambda} \int_{U} u_{\gamma, \psi} \varphi \,dx = \langle \psi, \varphi \rangle_{H^{-1,2}(U)} \quad \text{ for all $\varphi \in H^{1,2}_0(U)$}.
$$
As a direct consequence of the above, we have the following existence and uniqueness result.
\begin{proposition}
Assume {\bf (Y1)} and let $N \geq 0$ be a constant satisfying \eqref{vectorfieh}. Let $f \in L^{\frac{2d}{d+2}}(U)$.
Then, there exists a unique weak solution $u \in H^{1,2}_0(U)$ to
\begin{equation} \label{maineq3}
\left\{
\begin{alignedat}{2}
-{\rm div}(A \nabla u) + \langle \mathbf{H}, \nabla u \rangle + \left(c +\frac{N^2}{\lambda}\right) u &=f  && \quad \; \mbox{\rm in $U$,}\\
u &= 0 &&\quad \; \mbox{\rm on $\partial U$}.
\end{alignedat} \right.
\end{equation}
\end{proposition}
A disadvantage of the direct approach to existence and uniqueness via the Lax–Milgram theorem is that we only know the well-posedness of \eqref{maineq3}, not \eqref{maineq2}. Moreover, we cannot explicitly compute or characterize the number $N \ge 0$, and hence the zero-order coefficient must be taken as $c + \frac{N^2}{\lambda}$ rather than $c$. Of course, if we restrict attention to the case where ${\rm div}\mathbf{H} \leq 0$ and $c \geq 0$, then there is no need to perturb $c$ by $\frac{N^2}{\lambda}$. However, for a general vector field $\mathbf{H}$ without the assumption ${\rm div}\mathbf{H} \leq 0$, one needs an argument going beyond the Lax–Milgram theorem in order to establish existence and uniqueness of solutions to \eqref{maineq2}. Indeed, by combining a weak maximum principle with the Fredholm alternative, one can still treat zero-order coefficients $c \ge 0$ together with a general drift $\mathbf{H}$. The weak maximum principle for \eqref{maineq2} is stated as follows.

\begin{proposition} \label{weakmaxi}
Under the assumption {\bf (Y1)}, let $u \in H^{1,2}_0(U)$ be such that
\begin{equation*}
\int_U \left( \langle A \nabla u, \nabla \psi \rangle + \langle \mathbf{H}, \nabla u \rangle \psi + c u \psi \right) \, dx \leq 0 \quad \text{ for all } \psi \in H^{1,2}_0(U) \text{ with } \psi \geq 0.
\end{equation*}
Then $u \leq 0$ in $U$.
\end{proposition}
Indeed, the original proof of Proposition \ref{weakmaxi} is presented in \cite[Theorem 1]{Tma} (cf. \cite{T73}), and an alternative approach to the proof of Proposition \ref{weakmaxi} is given in \cite[Theorem 2.1.8 and Lemma 2.1.7]{BKRS15}. As a corollary of Proposition \ref{weakmaxi}, we immediately obtain the following.
\begin{corollary} \label{uniquenessres}
Assume {\bf (Y1)}. If $u$ is a weak solution to \eqref{maineq2} with $f=0$, then $u=0$ in $U$.
\end{corollary}

Now, under the assumption \textbf{(Y1)}, let $N \geq 0$ be the constant as in \eqref{vectorfieh} and $\gamma:=\frac{N^2}{\lambda}$. Let $\mathcal{B}$ be the bilinear form defined as in \eqref{bilinearf}.
By the Lax–Milgram theorem(\cite[Corollary 5.8]{Br11}) with \eqref{boundedness} and \eqref{desiredestim2}, for each $\psi \in H^{-1,2}(U)$ there exists a unique $u_{\gamma, \psi} \in H^{1,2}_0(U)$ such that
\begin{equation} \label{variationiden}
\mathcal{B}(u_{\gamma, \psi}, \varphi)  + \gamma \int_{U} u_{\gamma, \psi} \varphi\,dx = \langle \psi, \varphi \rangle_{H^{-1,2}(U)} \quad \text{ for all $\varphi \in H^{1,2}_0(U)$}.
\end{equation}
Substituting $u_{\gamma, \psi}$ for $\varphi$ in \eqref{variationiden} and using \eqref{desiredestim2} and the H\"{o}lder inequality,
\begin{align*}
\frac{\lambda}{2} \| \nabla u_{\gamma, \psi} \|^2_{L^2(U)} \leq  \langle \psi, u_{\gamma, \psi} \rangle_{H^{-1,2}(U)} \leq \|\psi \|_{H^{-1,2}(U)} \| \nabla u_{\gamma, \psi} \|_{L^{2}(U)},
\end{align*}
and hence we obtain that
\begin{equation} \label{energyestima}
\| \nabla u_{\gamma, \psi} \|_{L^2(U)} \leq \frac{2}{\lambda} \| \psi\|_{H^{-1,2}(U)}.
\end{equation}
Define the operator $K:H^{-1,2}(U) \rightarrow H^{1,2}_0(U)$ by setting
\begin{equation} \label{defnofk}
K \psi:= u_{\gamma, \psi}, \quad \psi \in H^{-1,2}(U).
\end{equation}
Consequently, from \eqref{energyestima} and \eqref{defnofk}, the boundedness of $K$ is established:
$$
\| K \psi \|_{H^{1,2}_0(U)} \leq \frac{2}{\lambda} \| \psi\|_{H^{-1,2}(U)} \quad \text{ for all } \psi \in H^{-1,2}(U).
$$
We introduce the operator $J: H^{1,2}_0(U) \to H^{-1,2}(U)$, where for each $u \in H^{1,2}_0(U)$, the action of $J(u)$ is given by
\begin{equation} \label{inclusionmap}
\langle J(u), \varphi \rangle_{H^{-1,2}(U)} = \int_{U} u \varphi \,dx  \quad \text{ for all $\varphi \in H^{1,2}_0(U)$}.
\end{equation}
The Rellich–Kondrachov compactness theorem (\cite[Section 5.7, Theorem 1]{E10}) and \eqref{inclusionmap} ensure that $J$ is a compact operator, and thus $K \circ J$ is also compact. Then, one can show  the following assertion. \\ \\
{\sf \underline{Assertion}:} Let $u \in H_0^{1,2}(U)$ and $\psi \in H^{-1,2}(U)$. The following statements {\rm ($\alpha$)--($\beta$)} are equivalent:
\begin{equation} \label{bvpequival}
\left\{
\begin{aligned}
    u- \gamma \big( K \circ J  \big)u &= K \psi \quad \text{ in } H^{1,2}_0(U). &&  \;\;\qquad {\rm (\alpha)} \\[1em]
    \mathcal{B}(u, \varphi)&= \langle \psi, \varphi \rangle_{H^{-1,2}(U)} \qquad \text{ for all $\varphi \in H^{1,2}_0(U)$.} && \;\;\qquad {\rm (\beta)}
\end{aligned}
\right.
\end{equation}
To verify the assertion, note that if ($\alpha$) holds, then
\[
u = K\big(\psi+\gamma J(u)\big).
\]
By the definition of $K$ (see \eqref{defnofk} and \eqref{variationiden}), this equality is equivalent to
\begin{equation}\label{variatcondia}
\mathcal{B}_{\gamma}(u,\varphi)
=\big\langle \psi+\gamma J(u),\varphi\big\rangle_{H^{-1,2}(U)}
\quad \text{for all }\varphi\in H^{1,2}_0(U),
\end{equation}
and, in view of \eqref{inclusionmap}, this is exactly ($\beta$). Conversely, if ($\beta$) holds, then \eqref{variatcondia} follows. By the definition of $K$,
\[
u = K\big(\psi+\gamma J(u)\big),
\]
and hence ($\alpha$) holds. This completes the proof of the assertion.
\\
Now, let $I: H^{1,2}_0(U) \to H^{1,2}_0(U)$ be the identity operator. Setting $\psi = 0$ in statement ($\alpha$), the equivalence established in \eqref{bvpequival} indicates that
\[
\text{$u \in H^{1,2}_0(U)$ satisfies } \big(I - \gamma (K \circ J) \big) u = 0
\]
if and only if
\[
\mathcal{B}(u, \varphi) = 0 \quad \text{for all } \varphi \in H^{1,2}_0(U).
\]
Hence, by Corollary \ref{uniquenessres}, it follows that the kernel of the operator is trivial:
\[
\left\{ u \in H^{1,2}_0(U) : \big(I - \gamma (K \circ J)\big) u = 0 \right\} = \{0\}.
\]
Given that $\gamma (K \circ J): H^{1,2}_0(U) \to H^{1,2}_0(U)$ is a compact operator, the Fredholm alternative (see \cite[Theorem~6.6]{Br11}) guarantees that for every $\psi \in H^{-1,2}(U)$, there exists a solution $u_{\psi} \in H^{1,2}_0(U)$ to
\[
\big(I - \gamma (K \circ J)\big) u_\psi = K \psi.
\]
Therefore, by the equivalence demonstrated in \eqref{bvpequival}, we obtain the following existence and uniqueness result.
\begin{theorem} \label{submaintheor}
Assume {\bf (Y1)} and let $f \in L^{\frac{2d}{d+2}}(U)$. Then, there exists a unique weak solution $u$ to \eqref{maineq2}.
\end{theorem}

\section{Divergence-free transformation} \label{sec4}
Before studying the divergence-free transformation methods, let us first consider the following simple homogeneous boundary value problem:
\begin{equation} \label{maineqsimp}
\left\{
\begin{alignedat}{2}
-\Delta u + c u &= f  && \quad \; \mbox{\rm in $U$,}\\
u &= 0 &&\quad \; \mbox{\rm on $\partial U$},
\end{alignedat} \right.
\end{equation}
where $c \in L^{\frac{2d}{d+2}}(U)$ with $c \ge 0$ and $f \in L^{\frac{2d}{d+2}}(U)$. A previous approach via bilinear form methods does not directly apply to prove the existence and uniqueness of weak solutions to \eqref{maineqsimp}, since the coefficient $c$ has low integrability, namely $c \in L^{\frac{2d}{d+2}}(U)$. On the other hand, we can easily derive existence and uniqueness for solutions to \eqref{maineqsimp} via a compactness argument. To this end, let $c_n := c \wedge n \in L^{\infty}(U)$. Then, by Theorem \ref{submaintheor}, there exists a unique function $u_n \in H^{1,2}_0(U)$ such that
\begin{equation} \label{weaksolutio}
\int_{U} \langle \nabla u_n, \nabla \varphi \rangle\,dx
+ \int_{U} c_n u_n \varphi\,dx
= \int_{U} f \varphi \,dx,
\quad \text{for all $\varphi \in H_0^{1,2}(U)$}.
\end{equation}
Substituting $\varphi = u_n$ in \eqref{weaksolutio} and using Sobolev's inequality, we obtain
\begin{equation} \label{energestiapp}
\| \nabla u_n \|_{L^2(U)}
\le \frac{2(d-1)}{d-2} \, \|f\|_{L^{\frac{2d}{d+2}}(U)}.
\end{equation}
Then, applying a weak compactness argument to \eqref{weaksolutio}, we obtain $u \in H^{1,2}_0(U)$ such that
\begin{equation} \label{weaksolutexistu}
\int_{U} \langle \nabla u, \nabla \varphi \rangle\,dx
+ \int_{U} c u \varphi\,dx
= \int_{U} f \varphi \,dx,
\quad \text{for all $\varphi \in C_0^{\infty}(U)$}
\end{equation}
and
\begin{equation} \label{energest}
\| \nabla u \|_{L^2(U)}
\le \frac{2(d-1)}{d-2} \, \|f\|_{L^{\frac{2d}{d+2}}(U)}.
\end{equation}
Moreover, let us additionally assume that $f \in L^{\frac{pd}{p+d}}(U)$ for some $p \in (d, \infty)$. Then, by the standard Moser iteration argument (see \cite[Proof of Theorem 4.2]{L24}), the function $u_n \in H^{1,2}_0(U)$ constructed above actually belongs to $H^{1,2}_0(U) \cap L^{\infty}(U)$, and the following estimate holds:
\begin{equation} \label{applinftboun}
\|u_n \|_{L^{\infty}(U)} \leq C \|f \|_{L^{\frac{pd}{d+p}}(U)},
\end{equation}
where $C>0$ is a constant which depends only on $d$, $p$ and $|U|$ (and is independent of $c$ and $n$). By weak and weak-* compactness applied to \eqref{weaksolutio}, we obtain $u \in H^{1,2}_0(U) \cap L^{\infty}(U)$ such that
\begin{equation} \label{linfintbdd}
\|u \|_{L^{\infty}(U)} \leq C \|f \|_{L^{\frac{pd}{d+p}}(U)}.
\end{equation}
To show uniqueness of weak solutions to \eqref{maineqsimp}, the previous methods based on the weak maximum principle are no longer applicable because of the low integrability of $c$,  namely $c \in L^{\frac{2d}{d+2}}(U)$.  Instead, we use a duality argument inspired by \cite[Lemma 2.2.11]{K07}. Let $v \in H^{1,2}_0(U)$ be such that
\begin{equation} \label{weaksdfirsi}
\int_{U} \langle \nabla v, \nabla \varphi \rangle \,dx + \int_{U} c v \varphi\,dx = 0, \quad \text{ for all $\varphi \in C_0^{\infty}(U)$}.
\end{equation}
To verify uniqueness, it is enough to show that $v=0$. To apply the duality argument, we now fix $\phi \in C_0^{\infty}(U)$. Then, by the existence and regularity result above, there exists $w \in H^{1,2}_0(U) \cap L^{\infty}(U)$ such that
\begin{equation} \label{weaksdualexi}
\int_{U} \langle \nabla w, \nabla \varphi \rangle\,dx
+ \int_{U} c w \varphi\,dx
= \int_{U} \phi \varphi\,dx,
\quad \text{for all $\varphi \in C_0^{\infty}(U)$}.
\end{equation}
By the denseness of $C_0^{\infty}(U)$ in $H^{1,2}_0(U)$, we have 
\begin{equation} \label{weaksdense}
\int_{U} \langle \nabla w, \nabla v \rangle\,dx
+ \int_{U} c w v\,dx
= \int_{U} \phi v\,dx.
\end{equation}
Since $w \in H^{1,2}_0(U) \cap L^{\infty}(U)$, by \cite[Proposition A.8]{L25jm}, there exist a constant $M>0$ and a sequence of functions $(\varphi_n)_{n \geq 1} \subset C_0^{\infty}(U)$ such that
\begin{equation} \label{densenes}
\sup_{n \geq 1} \| \varphi_n \|_{L^{\infty}(U)} \leq M, \quad \lim_{n \rightarrow \infty} \varphi_n = w \;\; \text{ in $H^{1,2}_0(U)$}, \;\; \lim_{n \rightarrow \infty} \varphi_n = w \; \text{ a.e. in $U$}.
\end{equation}
Then, applying \eqref{densenes} to \eqref{weaksdfirsi} and using Lebesgue's dominated convergence theorem, we obtain
\begin{equation} \label{weaksolfini}
\int_{U} \langle \nabla v, \nabla w \rangle \,dx + \int_{U} c v w\,dx = 0.
\end{equation}
Comparing \eqref{weaksdense} and \eqref{weaksolfini}, we get
\begin{equation} \label{viszero}
\int_{U} \phi v\,dx = 0.
\end{equation}
Since $\phi$ is an arbitrary function in $C_0^{\infty}(U)$, we finally obtain $v=0$, and hence uniqueness follows.\\
The above argument can be summarized as follows.
\begin{theorem} \label{l2lowerstarze}
Let $U$ be a bounded open subset of $\mathbb{R}^d$ with $d \geq 3$.
Assume $c \in L^{\frac{2d}{d+2}}(U)$ with $c \ge 0$. Then the following hold.
\begin{itemize}
\item[(i)]
If $v \in H^{1,2}_0(U)$ satisfies \eqref{weaksdfirsi}, then $v=0$ in $U$.

\item[(ii)]
If $f \in L^{\frac{2d}{d+2}}(U)$, then there exists a weak solution $u \in H^{1,2}_0(U)$ to \eqref{maineqsimp} such that \eqref{energest} holds.

\item[(iii)]
If $f \in L^{\frac{\hat{p}d}{d+\hat{p}}}(U)$ for some $\hat{p} \in (d, \infty)$, then the (unique) weak solution $u$ to \eqref{maineqsimp} as in \textnormal{(ii)} satisfies $u \in L^{\infty}(U)$ and
\eqref{linfintbdd}.
\end{itemize}
\end{theorem}
As far as we know, the above well-posedness result under the minimal condition $c\in L^{\frac{2d}{d+2}}(U)$ does not seem to be explicitly available in the literature. In fact, we would like to generalize Theorem \ref{l2lowerstarze} from the case $c \in L^{\frac{2d}{d+2}}(U)$ to the more general assumption $c \in L^1(U)$. To this end, we again employ an approximation argument. Assume that $f \in L^{\frac{\hat{p}d}{d+\hat{p}}}(U)$ for some $\hat{p} \in (d, \infty)$ and $c \in L^1(U)$ with $c \geq 0$ in $U$. Let $c_n := c \wedge n$. Then, by Theorem \ref{l2lowerstarze}, there exists $u_n \in H^{1,2}_0(U) \cap L^{\infty}(U)$ such that \eqref{weaksolutio},
\eqref{energestiapp} and \eqref{applinftboun} hold. By weak and weak-* compactness applied to \eqref{weaksolutio}, we obtain $u \in H^{1,2}_0(U) \cap L^{\infty}(U)$ such that \eqref{weaksolutexistu}, \eqref{energest} and \eqref{linfintbdd} hold. Next, we prove the uniqueness result. Assume that $v \in H^{1,2}_0(U)$ satisfies $cv \in L^1(U)$ and \eqref{weaksdfirsi}. To show $v=0$ in $U$, let $\phi \in C_0^{\infty}(U)$ be arbitrary. By the existence and regularity result, there exists $w \in H^{1,2}_0(U) \cap L^{\infty}(U)$ such that \eqref{weaksdualexi} holds.
For each $n \geq 1$, define $v_n := (v \wedge n) \vee (-n)$. By the denseness result in \eqref{densenes} and Lebesgue's dominated convergence theorem, we have
\begin{equation} \label{weaksdeappron}
\int_{U} \langle \nabla w, \nabla v_n \rangle\,dx + \int_{U} c w v_n\,dx
= \int_{U} \phi v_n\,dx, \quad \text{for each $n \geq 1$}.
\end{equation}
Observe that
$$
|c w v_n|
= |c|\,|v_n|\,|w|
\leq |c|\,|v|\,|w|
= |cv|\,|w|
\quad \text{for each $n \geq 1$},
$$
and, by \cite[Chapter I, Proposition 4.17]{MR92},
$$
\lim_{n \rightarrow \infty} v_n = v \;\; \text{in $H^{1,2}_0(U)$} 
\quad \text{and} \quad
\lim_{n \rightarrow \infty} v_n = v \;\; \text{a.e. in $U$}.
$$
Therefore, passing to the limit $n \rightarrow \infty$ in \eqref{weaksdeappron} and using Lebesgue's dominated convergence theorem, we obtain
\begin{equation} \label{weaksorigi}
\int_{U} \langle \nabla w, \nabla v \rangle\,dx + \int_{U} c w v\,dx
= \int_{U} \phi v\,dx.
\end{equation}
Meanwhile, applying the denseness result as in \eqref{densenes} to \eqref{weaksdfirsi}, we get \eqref{weaksolfini}.
Comparing \eqref{weaksorigi} and \eqref{weaksolfini}, we obtain \eqref{viszero}, and hence $v=0$, as desired. \\
Summarizing, we obtain the following result.
\begin{theorem} \label{theoelone}
Let $U$ be a bounded open subset of $\mathbb{R}^d$ with $d \geq 3$.
Assume that $c \in L^{1}(U)$ with $c \ge 0$. Then the following hold.
\begin{itemize}
\item[(i)]
If $v \in H^{1,2}_0(U)$ with $cv \in L^1(U)$ satisfies \eqref{weaksdfirsi},
then $v=0$ in $U$.
\item[(ii)]
If $f \in L^{\frac{\hat{p}d}{d+\hat{p}}}(U)$ for some $\hat{p} \in (d, \infty)$, then there exists a weak solution $u \in H^{1,2}_0(U) \cap L^{\infty}(U)$ to \eqref{maineqsimp} such that \eqref{energest} and \eqref{linfintbdd} hold.
\end{itemize}
\end{theorem}

As a direct consequence of Theorem \ref{theoelone}, we obtain a nontrivial existence and uniqueness result for strong solutions to \eqref{maineqsimp}. Assume that $U$ is a bounded $C^{1,1}$ domain in $\mathbb{R}^d$ with $d \geq 3$. 
Let $c \in L^s(U)$ for some $s \in (1,\infty)$ with $c \ge 0$ in $U$, and let $f \in L^{\frac{\hat{p}d}{d+\hat{p}}}(U)$ for some $\hat{p} \in (d,\infty)$. Here we call $u$ a strong solution to \eqref{maineqsimp} if $u \in H^{1,2}_0(U) \cap H^{2,1}(U)$, $cu \in L^1(U)$, and
\[
-\Delta u(x) + c u(x) = f(x) \quad \text{for a.e. } x \in U.
\]
Since every strong solution to \eqref{maineqsimp} is also a weak solution to \eqref{maineqsimp}, the uniqueness of strong solutions follows from Theorem~\ref{theoelone}(i). Moreover, by Theorem~\ref{theoelone}(ii) there exists $u \in H^{1,2}_0(U) \cap L^{\infty}(U)$ such that \eqref{weaksolutexistu} holds, that is,
\[
\int_U \langle \nabla u, \nabla \varphi \rangle \,dx = \int_U (f - cu)\varphi \,dx \quad \text{for all } \varphi \in C_0^{\infty}(U).
\]
On the other hand, by \cite[Theorem~2.2]{KK19} there exists $\tilde{u} \in H^{2,s}(U) \cap H^{1,s}_0(U)$ such that
\[
\int_U \langle \nabla \tilde{u}, \nabla \varphi \rangle \,dx = \int_U (f - cu)\varphi \,dx \quad \text{for all } \varphi \in C_0^{\infty}(U).
\]
Finally, using the uniqueness result in $H^{1,s \wedge 2}_0(U)$ from \cite[Theorem~2.1(i)]{KK19}, we obtain $u = \tilde{u} \in H^{1,2}_0(U) \cap H^{2,s}(U)$. Summarizing, this yields strong well-posedness for \eqref{maineqsimp} under the substantially weaker condition $c \in L^s(U)$ with $s>1$, compared to the classical assumption $c \in L^{\frac{d}{2}}(U)$.

\begin{theorem}
Assume that $U$ is a bounded $C^{1,1}$ domain in $\mathbb{R}^d$ with $d \geq 3$. Let $c \in L^s(U)$ for some $s \in (1,\infty)$ with $c \ge 0$ in $U$, and let $f \in L^{\frac{\hat{p}d}{d+\hat{p}}}(U)$ for some $\hat{p} \in (d,\infty)$. Then there exists a unique strong solution $u$ to \eqref{maineqsimp}. In particular,
\[
u \in H^{2,s}(U) \cap H^{1,2}_0(U) \cap L^{\infty}(U).
\]
\end{theorem}

Furthermore, the results of Theorems \ref{l2lowerstarze} and \ref{theoelone} can be extended to linear elliptic equations with a drift term whose weak divergence is non-positive. The idea of the proof is similar to that used for \eqref{l2lowerstarze} and \eqref{theoelone}, but one has to apply more delicate approximations (see \cite{L24}).

\begin{theorem} \label{divfreeltwore}
Let $U$ be a bounded open subset of $\mathbb{R}^d$ with $d \geq 3$.
Let $\hat{A}=(\hat{a}_{ij})_{1 \leq i,j \leq d}$ be a matrix of functions such that $\hat{A}$ is uniformly strictly elliptic and bounded (see \eqref{elliptici}), and let $\hat{\mathbf{B}} \in L^2(U,\mathbb{R}^d)$ satisfy 
\begin{equation} \label{negativweakdiv}
\int_{U} \langle \hat{\mathbf{B}}, \nabla \varphi \rangle\,dx \geq 0, \quad \text{for all $\varphi \in C_0^{\infty}(U)$ with $\varphi \geq 0$}.
\end{equation}
Assume that $\hat{c} \in L^{\frac{2d}{d+2}}(U)$ with $\hat{c} \ge 0$. 
Then the following statements hold.
\begin{itemize}
\item[(i)]
If $\hat{v} \in H^{1,2}_0(U)$ satisfies
\begin{equation} \label{vfununiqdiv}
\int_{U} \langle \hat{A} \nabla \hat{v}, \nabla \varphi \rangle \,dx 
+ \int_{U} \langle \hat{\mathbf{B}}, \nabla \hat{v} \rangle \varphi\,dx
+ \int_{U} \hat{c} \hat{v} \varphi\,dx = 0, \quad \text{for all $\varphi \in C_0^{\infty}(U)$},
\end{equation}
then $\hat{v}=0$ in $U$.

\item[(ii)]
If $\hat{f} \in L^{\frac{2d}{d+2}}(U)$, then there exists a weak solution $\hat{u} \in H^{1,2}_0(U)$ to 
\begin{equation} \label{maineqhat}
\left\{
\begin{alignedat}{2}
-{\rm div}(\hat{A} \nabla \hat{u}) 
+ \langle \hat{\mathbf{B}}, \nabla \hat{u} \rangle 
+ \hat{c} \hat{u} &= \hat{f}  && \quad \mbox{\rm in $U$,}\\
\hat{u} &= 0 &&\quad \mbox{\rm on $\partial U$}.
\end{alignedat} \right.
\end{equation}
such that 
\begin{equation} \label{energesthat}
\| \nabla \hat{u} \|_{L^2(U)}
\le \hat{C}_1 \|\hat{f}\|_{L^{\frac{2d}{d+2}}(U)},
\end{equation}
where $\hat{C}_1>0$ is a constant depending only on $d$ and $\lambda$.

\item[(iii)]
If $\hat{f} \in L^{\frac{\hat{p}d}{d+\hat{p}}}(U)$ for some $\hat{p} \in (d, \infty)$, then the (unique) weak solution $\hat{u}$ to \eqref{maineqhat} obtained in \textnormal{(ii)} satisfies $\hat{u} \in L^{\infty}(U)$ and
\begin{equation} \label{linftybdhat}
\|\hat{u}\|_{L^{\infty}(U)} \leq \hat{C}_2 \| \hat{f} \|_{L^{\frac{\hat{p}d}{d+\hat{p}}}(U)},
\end{equation}
where $\hat{C}_2>0$ is a constant depending only on $\lambda$, $d$, $\hat{p}$, and $|U|$.
\end{itemize}
\end{theorem}

Using an approximation argument for Theorem \ref{divfreeltwore}, analogous to that in the proof of Theorem \ref{theoelone},  we obtain the following result.

\begin{theorem} \label{divfreqelone}
Let $U$ be a bounded open subset of $\mathbb{R}^d$ with $d \geq 3$.
Let $\hat{A}=(\hat{a}_{ij})_{1 \leq i,j \leq d}$ be a matrix of functions such that $\hat{A}$ is uniformly strictly elliptic and bounded (see \eqref{elliptici}), and let $\hat{\mathbf{B}} \in L^2(U,\mathbb{R}^d)$ satisfy \eqref{negativweakdiv}.
Assume that $\hat{c} \in L^{1}(U)$ with $\hat{c} \ge 0$. 
Then the following hold.
\begin{itemize}
\item[(i)]
If $\hat{v} \in H^{1,2}_0(U)$ satisfies $\hat{c} \hat{v} \in L^1(U)$ and \eqref{vfununiqdiv},
then $\hat{v}=0$ in $U$.

\item[(ii)]
If $\hat{f} \in L^{\frac{\hat{p}d}{d+\hat{p}}}(U)$ for some $\hat{p} \in (d, \infty)$, then there exists a (unique) weak solution $\hat{u} \in H^{1,2}_0(U)$ to \eqref{maineqhat} such that 
\eqref{energesthat} and \eqref{linftybdhat} hold.
\end{itemize}
\end{theorem}

We now consider the following basic structural condition, under which we can apply the divergence-free transformation:\\ \\
{\bf (Y2)}:
\textit{
$U$ is a bounded open subset of $\mathbb{R}^d$ with $d \ge 3$, $\mathbf{H} \in L^{p}(U, \mathbb{R}^d)$ for some $p \in (d, \infty)$, and $A = (a_{ij})_{1 \le i,j \le d}$ is a (possibly non-symmetric) matrix of measurable functions on $\mathbb{R}^d$ such that $A$ is uniformly strictly elliptic and bounded (see \eqref{elliptici}).
}
\centerline{}
\centerline{}
The main ingredient of our divergence-free transformation is the following result, which is a part of \cite[Theorem~3.1]{L25jm}. The motivation comes from \cite{LST22} and \cite{LT21}, and closely related results are due to \cite{BRS12} and \cite{BKRS15}. In particular, in stochastic analysis it is closely related to the existence of infinitesimally invariant measures.  However, the divergence-free transformation formulated in Theorem~\ref{divfretrans} does not seem to be contained in these works and appears to be new.

\begin{theorem}{\cite[Theorem 3.1]{L25jm}} \label{existinvmea}
Under the assumption {\bf (Y2)}, there exists a function $\rho \in H^{1,2}(U) \cap C(\overline{U})$ and a point $x_1 \in U$ such that $\rho(x)>0$ for all $x \in \overline{U}$ and $\rho(x_1)=1$, and
\begin{equation} \label{infinva}
\int_{U} \langle A^T \nabla \rho  + \rho \mathbf{H}, \nabla \varphi \rangle \,dx = 0 
\quad \text{for all $\varphi \in C_0^{\infty}(U)$}.
\end{equation}
Moreover, there exists a constant $K_1 \geq 1$, depending only on $d$, $\lambda$, $M$, $U$, $p$, and $\|\mathbf{H}\|_{L^p(U)}$, such that
$$
1 \leq \max_{\overline{U}} \rho  \leq K_1 \min_{\overline{U}} \rho \leq K_1.
$$
\end{theorem}

Now, using $\rho$ constructed in Theorem \ref{existinvmea}, we are ready to derive the divergence-free transformation. Let $\rho \in H^{1,2}(U) \cap C(\overline{U})$ be a function constructed in Theorem \ref{existinvmea} and define 
\begin{equation} \label{mathbfbconst}
\mathbf{B}:=\mathbf{H}+\frac{1}{\rho} A^T \nabla \rho \quad \text{on $U$}.
\end{equation}
Then, $\rho \mathbf{B} \in L^2(U,\mathbb{R}^d)$ and by \eqref{infinva},
$$
\int_{U} \langle \rho \mathbf{B}, \nabla \varphi \rangle \,dx = 0 \quad \text{for all $\varphi \in C_0^{\infty}(U)$}.
$$
Now, let $f \in L^1(U)$, $u \in H^{1,2}_0(U)$, and $c \in L^1(U)$ with $cu \in L^1(U)$. Assume that $u$ is a weak solution to \eqref{maineq2}, i.e., \eqref{weakformu} is satisfied. Then, by an approximation argument as in \eqref{densenes}, we get
\begin{equation}  
\begin{aligned}[b]
&\int_U \langle A \nabla u, \nabla \psi \rangle + \langle \mathbf{H}, \nabla u \rangle \psi + cu \psi \,dx \\
&\;\;= \int_U f \psi \,dx \quad \text{for all $\psi \in H_0^{1,2}(U) \cap L^{\infty}(U)$}. \label{weakformuappr}
\end{aligned}
\end{equation}
Now, let $\varphi \in C_0^{\infty}(U)$ be arbitrarily given. Since $\rho \in H^{1,2}(U) \cap L^{\infty}(U)$, by the product rule (cf. \cite[Chapter I, Corollary 4.15]{MR92}), we have $\varphi \rho \in H^{1,2}_0(U) \cap L^{\infty}(U)$. Substituting $\varphi \rho$ for $\psi$ in \eqref{weakformuappr}, we get
\begin{align*}
&\int_U f \varphi \rho \,dx 
 = \int_U \langle A \nabla u, \nabla (\rho \varphi) \rangle  +  \langle \mathbf{H}, \nabla u \rangle \rho \varphi    +  cu \rho \varphi \,dx \\
 &= \int_{U} \langle \rho A \nabla u, \nabla  \varphi \rangle   + \left \langle \frac{1}{\rho}A^T \nabla \rho + \mathbf{H}, \nabla u \right \rangle \rho \varphi   + \rho cu \varphi  \,dx \\
 &= \int_U \langle \rho A \nabla u, \nabla \varphi \rangle + \langle \rho \mathbf{B}, \nabla u \rangle \varphi +  \rho cu \varphi \,dx.
\end{align*}
Hence, by an approximation argument as in \eqref{densenes}, we obtain 
\begin{equation} 
\begin{aligned}[b] 
 &\int_U \langle \rho A \nabla u, \nabla \varphi \rangle + \langle \rho \mathbf{B}, \nabla u \rangle \varphi +  \rho cu \varphi \,dx \\
 &\;\;= \int_{U} \rho f \varphi\,dx \quad \text{for all $\varphi \in H_0^{1,2}(U) \cap L^{\infty}(U)$}. \label{weakfordivfre}
\end{aligned}
\end{equation}
Conversely, assume that $u$ is a weak solution to
\begin{equation} \label{maineqdivfree}
\left\{
\begin{alignedat}{2}
-{\rm div}(\rho A \nabla u) + \langle \rho \mathbf{B}, \nabla u \rangle + \rho c u &=\rho f  && \quad \; \mbox{\rm in $U$,}\\
u &= 0 &&\quad \; \mbox{\rm on $\partial U$}.
\end{alignedat} \right.
\end{equation}
Then, by an approximation argument as in \eqref{densenes}, $u$ satisfies \eqref{weakfordivfre}. Let $\psi \in C_0^{\infty}(U)$ be arbitrarily given. 
Since $\frac{1}{\rho} \in H^{1,2}(U) \cap L^{\infty}(U)$, by the product rule, we have $\frac{\psi}{\rho} \in H^{1,2}_0(U) \cap L^{\infty}(U)$. Substituting $\frac{\psi}{\rho}$ for $\varphi$ in \eqref{weakfordivfre}, we get
\begin{align*}
 &\int_{U} f \psi  \,dx 
 = \int_U \left \langle \rho A \nabla u, \nabla \left(  \frac{\psi}{\rho} \right) \right\rangle  + \langle \mathbf{B}, \nabla u \rangle \psi   +cu \psi \,dx\\
 & = \int_{U} \langle A \nabla u, \nabla \psi \rangle +  \left \langle -\frac{1}{\rho}A^T \nabla \rho + \mathbf{B}, \nabla u \right \rangle  \psi + cu \psi \,dx  \\[0.2em]
 & = \int_U \langle A \nabla u, \nabla \psi \rangle + \langle \mathbf{H}, \nabla u \rangle \psi  +  cu \psi \,dx.
\end{align*}
Hence, by an approximation argument as in \eqref{densenes}, we obtain \eqref{weakformuappr}. Summarizing the above, we obtain the following divergence-free transformation:

\begin{theorem}  {\bf (Divergence-free transformation)} {\cite[Theorem 3.2]{L25jm}} \label{divfretrans}\\
Let $\rho$ be the function constructed in Theorem \ref{existinvmea} and $\mathbf{B}$ be a vector field defined as in \eqref{mathbfbconst}.
Let $f \in L^1(U)$, $u \in H^{1,2}_0(U)$, and $c \in L^1(U)$ with $cu \in L^1(U)$. Then $u$ satisfies \eqref{weakformuappr} if and only if $u$ satisfies
\eqref{weakfordivfre}. In other words, $u$ is a weak solution to \eqref{maineq2} if and only if $u$ is a weak solution to \eqref{maineqdivfree}.
\end{theorem}

By using the divergence-free transformation above, we obtain the following direct consequence of Theorems \ref{divfreeltwore} and \ref{divfreqelone}.

\begin{theorem} \label{maintheojma}
Assume {\bf (Y2)} and let $c \in L^1(U)$ with $c \geq 0$ in $U$.
Then the following statements hold:
\begin{itemize}
\item[(i)]
If $v \in H^{1,2}_0(U)$ with $cv \in L^1(U)$ satisfies
$$
\int_U \langle A \nabla v, \nabla \psi \rangle + \langle \mathbf{H}, \nabla v \rangle \psi  +  cv \psi \,dx=0, \quad \text{for all $\psi \in C_0^{\infty}(U)$},
$$
then $v=0$ in $U$.

\item[(ii)]
If $c \in L^{\frac{2d}{d+2}}(U)$ and $f \in L^{\frac{2d}{d+2}}(U)$, then there exists $u \in H^{1,2}_0(U)$ such that \eqref{weakformu} holds.
In particular, there exists a constant $\tilde{C}_1>0$, depending only on $d$, $\lambda$, $M$, and on the constant $K_1 \geq 1$ (see Theorem \ref{existinvmea}), such that
\begin{equation} \label{energestim}
\| \nabla u \|_{L^2(U)} \leq \tilde{C}_1 \|f\|_{L^{\frac{2d}{d+2}}(U)}.
\end{equation}
By Sobolev's inequality, $u \in L^{\frac{2d}{d-2}}(U)$ and satisfies
\begin{equation} \label{intebrbilit}
\|u\|_{L^{\frac{2d}{d-2}}(U)} \leq C_1 \|f\|_{L^{\frac{2d}{d+2}}(U)},
\end{equation}
where $C_1:=\frac{2(d-1)}{d-2} \tilde{C}_1$.

\item[(iii)]
If $f \in L^{\frac{\hat{p}d}{d+\hat{p}}}(U)$ for some $\hat{p} \in (d, \infty)$, then there exists $u \in H^{1,2}_0(U) \cap L^{\infty}(U)$ such that \eqref{weakformu}, \eqref{energestim}, and \eqref{intebrbilit} hold. Moreover, there exists a constant $C_2>0$, depending only on $d$, $\lambda$, $M$, $\hat{p}$, $|U|$, and on the constant $K_1 \geq 1$ (see Theorem \ref{existinvmea}), such that
$$
\| u \|_{L^\infty(U)} \leq C_2 \|f\|_{L^{\frac{\hat{p}d}{d+\hat{p}}}(U)}.
$$
\end{itemize}
\end{theorem}

\begin{remark}
Under the assumption $c \in L^1(U)$ in Theorem~\ref{maintheojma}, it remains open whether the existence and uniqueness of solutions to \eqref{maineq2} still hold when the condition $\mathbf{H} \in L^p(U,\mathbb{R}^d)$ for some $p \in (d,\infty)$ is relaxed to the borderline case $\mathbf{H} \in L^d(U,\mathbb{R}^d)$. In any case, it seems difficult to resolve this problem by using the divergence-free transformation method developed in \cite{L25jm}, as the Harnack inequality for $\rho$ may not be available when $\mathbf{H}\in L^{d}(U, \mathbb{R}^{d})$.
\end{remark}

\section{Interpolations} \label{sec5}
In Theorem \ref{maintheojma}, we considered two conditions on the zero-order coefficient $c$, namely $c \in L^{\frac{2d}{d+2}}(U)$ and $c \in L^1(U)$, where the corresponding assumptions on the data are $f \in L^{\frac{2d}{d+2}}(U)$ and $f \in L^{\frac{\hat{p}d}{d+\hat{p}}}(U)$ with $\hat{p} \in (d, \infty)$, respectively. In the next theorem, we interpolate between these two cases.
\begin{theorem}
Under the assumption {\bf (Y2)}, let $\hat{p} \in (d, \infty)$, $r \in [2, d]$ and define 
$$
k:=\frac{r(\hat{p}-2)}{2(\hat{p}-r)} \in [1, \infty).
$$
Let $f \in L^{\frac{rd}{d+r}}(U)$ and $c \in L^{\frac{2dk}{2dk-d+2}}(U)$ with $c \geq 0$. Then there exists $u \in H^{1,2}_0(U) \cap L^{\frac{2dk}{d-2}}(U)$ such that $u$ is the unique weak solution to \eqref{maineq2}. Moreover,
\begin{equation} \label{interpolineq}
\|u\|_{L^{\frac{2dk}{d-2}}(U)} \leq C_1^{\frac{1}{k}} C_2^{1-\frac{1}{k}} \|f\|_{L^{\frac{rd}{d+r}}(U)},
\end{equation}
where $C_1$ and $C_2>0$ are constants as in Theorem \ref{maintheojma}.
\end{theorem}
\begin{proof}
For each $n \geq 1$, let $f_n:=(f \wedge n) \vee (-n)$ and $c_n:=(c \wedge n) \vee (-n)$. Let $u_n \in H^{1,2}_0(U) \cap L^{\infty}(U)$ be such that
\begin{equation} \label{approxiequa}
\int_U \langle A \nabla u_n, \nabla \psi \rangle + \langle \mathbf{H}, \nabla u_n \rangle \psi  +  c_n u_n \psi \,dx
=\int_{U} f_n \psi\,dx, \quad \text{for all $\psi \in C_0^{\infty}(U)$},
\end{equation}
as in Theorem \ref{maintheojma}(iii). It then follows from Theorem \ref{maintheojma}(iii) that
\begin{align}
&	\| \nabla u_n \|_{L^2(U)} \leq \tilde{C}_1 \|f_n\|_{L^{\frac{2d}{d+2}}(U)} \leq \tilde{C}_1 \|f\|_{L^{\frac{2d}{d+2}}(U)}, \label{appenerg} \\
&\qquad \qquad \|u_n\|_{L^{\frac{2d}{d-2}}(U)} \leq C_1 \|f_n\|_{L^{\frac{2d}{d+2}}(U)}, \label{appintegra} \\
& \qquad \qquad \|u_n \|_{L^{\infty}(U)}	 \leq C_2 \|f_n\|_{L^{\frac{\hat{p}d}{d+\hat{p}}}(U)}. \label{appbddness}
\end{align}
For each fixed $n$, let $T_n$ denote the linear solution operator
$g \mapsto u_n$ associated with \eqref{approxiequa} (with coefficients $c_n$ fixed).
Then \eqref{appintegra} and \eqref{appbddness} can be written as
\[
\|T_n g\|_{L^{q_0}(U)} \le C_1 \|g\|_{L^{p_0}(U)},
\qquad
\|T_n g\|_{L^{\infty}(U)} \le C_2 \|g\|_{L^{p_1}(U)},
\]
where $q_0=\frac{2d}{d-2}$, $p_0=\frac{2d}{d+2}$, and $p_1=\frac{\hat p d}{d+\hat p}$.
Choose $\theta\in(0,1)$ such that $1-\theta=\frac1k$ (i.e., $\theta=1-\frac1k$).
By the Riesz--Thorin interpolation theorem (\cite[15, Chapter 2, Theorem 2.1]{SS11}), there exist exponents $p_\theta$ and $q_\theta$ such that
\[
\|T_n g\|_{L^{q_\theta}(U)} \le C_1^{1-\theta} C_2^{\theta} \|g\|_{L^{p_\theta}(U)},
\quad
\frac1{q_\theta}=\frac{1-\theta}{q_0},\quad
\frac1{p_\theta}=\frac{1-\theta}{p_0}+\frac{\theta}{p_1}.
\]
Since $1-\theta=\frac1k$, we have $q_\theta=\frac{q_0}{1-\theta}=\frac{2dk}{d-2}$ and
\[
\frac1{p_\theta}
=\frac{1}{k}\frac{d+2}{2d}+\Bigl(1-\frac1k\Bigr)\frac{d+\hat p}{\hat p d}
=\frac{d+r}{rd},
\]
which is equivalent to $p_\theta=\frac{rd}{d+r}$ and $k=\frac{r(\hat p-2)}{2(\hat p-r)}$.
Hence we obtain
\begin{equation}  \label{integratwodk}
\|u_n\|_{L^{\frac{2dk}{d-2}}(U)} \leq C_1^{\frac{1}{k}} C_2^{1-\frac{1}{k}} \|f_n\|_{L^{\frac{rd}{d+r}}(U)} \leq C_1^{\frac{1}{k}} C_2^{1-\frac{1}{k}} \|f\|_{L^{\frac{rd}{d+r}}(U)}. 
\end{equation}
By the weak compactness implied by \eqref{appenerg} and \eqref{integratwodk}, there exist $u \in H^{1,2}_0(U) \cap L^{\frac{2dk}{d-2}}(U)$ and a subsequence of $(u_n)_{n \geq 1}$, denoted again by $(u_n)_{n \geq 1}$, such that
$$
\lim_{n \rightarrow \infty} u_n = u \;\; \text{weakly in $H^{1,2}_0(U)$} \quad \text{and} \quad  \lim_{n \rightarrow \infty} u_n = u \;\; \text{weakly in $L^{\frac{2dk}{d-2}}(U)$}.
$$
Observe that the exponents $\frac{2dk}{d-2}$ and $\frac{2dk}{2dk-d+2}$ are H\"{o}lder conjugates. Thus, passing to the limit as $n \rightarrow \infty$ in \eqref{approxiequa} and \eqref{integratwodk}, we obtain
\eqref{weakformu} and \eqref{interpolineq}, respectively.
\end{proof}

\centerline{}


\centerline{}

\centerline{}
\centerline{}
\noindent
Haesung Lee\\
Department of Mathematics and Big Data Science,  \\
Kumoh National Institute of Technology, \\
Gumi, Gyeongsangbuk-do 39177, Republic of Korea \\
E-mail: fthslt@kumoh.ac.kr, \; fthslt14@gmail.com

\begin{thebibliography}{99}

\bibitem{BKRS15}
V.I. Bogachev, N.V. Krylov, M. R\"ockner, S. Shaposhnikov,
{\it Fokker--Planck--Kolmogorov Equations}, American Mathematical Society, 2015.

\bibitem{BRS12}
V.I. Bogachev, M. R\"ockner, S. Shaposhnikov,
{\it On positive and probability solutions of the stationary Fokker--Planck--Kolmogorov equation},
Dokl. Math. 85(3) (2012), 350--354.

\bibitem{Br11}
H. Brezis,
{\it Functional Analysis, Sobolev Spaces and Partial Differential Equations},
Springer, New York, 2011.

\bibitem{Dr}
J. Droniou,
{\it Non-coercive linear elliptic problems},
Potential Anal. 17(2) (2002), 181--203.

\bibitem{E10}
L.C. Evans,
{\it Partial Differential Equations}, 2nd ed.,
Graduate Studies in Mathematics, 19, American Mathematical Society, Providence, RI, 2010.

\bibitem{GGM13}
F. Giannetti, L. Greco, G. Moscariello,
{\it Linear elliptic equations with lower-order terms},
Differential Integral Equations 26(5--6) (2013), 623--638.

\bibitem{KK19}
B. Kang, H. Kim,
{\it On $L^p$-resolvent estimates for second-order elliptic equations in divergence form},
Potential Anal. 50(1) (2019), 107--133.

\bibitem{KPT23}
H. Kim, T. Phan, T.P. Tsai,
{\it On linear elliptic equations with drift terms in critical weak spaces},
arXiv:2312.11215.

\bibitem{K07}
M. Kontovourkis,
{\it On elliptic equations with low-regularity divergence-free drift terms and the steady-state Navier--Stokes equations in higher dimensions},
Ph.D. thesis, University of Minnesota, 2007.

\bibitem{K21}
N.V. Krylov,
{\it Elliptic equations with $VMO$, $a,b\in L_d$, and $c\in L_{d/2}$},
Trans. Amer. Math. Soc. 374(4) (2021), 2805--2822.

\bibitem{LT21}
H. Lee, G. Trutnau,
{\it Existence and regularity of infinitesimally invariant measures, transition functions and time-homogeneous It\^o-SDEs},
J. Evol. Equ. 21(1) (2021), 601--623.

\bibitem{LST22}
H. Lee, W. Stannat, G. Trutnau,
{\it Analytic Theory of It\^o-Stochastic Differential Equations with Non-smooth Coefficients},
SpringerBriefs in Probability and Mathematical Statistics,
Springer, Singapore, 2022.

\bibitem{L24}
H. Lee,
{\it On the contraction properties for weak solutions to linear elliptic equations with $L^2$-drifts of negative divergence},
Proc. Amer. Math. Soc. 152(5) (2024), 2051--2068.

\bibitem{L25jm}
H. Lee,
{\it Analysis of linear elliptic equations with general drifts and $L^1$-zero-order terms},
J. Math. Anal. Appl. 548 (2025), 129425.

\bibitem{MR92}
Z.M. Ma, M. R\"ockner,
{\it Introduction to the Theory of (Nonsymmetric) Dirichlet Forms},
Springer, Berlin, 1992.

\bibitem{S65}
G. Stampacchia,
{\it Le probl\`eme de Dirichlet pour les \'equations elliptiques du second ordre \`a coefficients discontinus},
Ann. Inst. Fourier (Grenoble) 15(1) (1965), 189--258.

\bibitem{SS11}
E.M. Stein, R. Shakarchi,
{\it Functional Analysis},
Princeton Lectures in Analysis, 4, Princeton University Press, Princeton, NJ, 2011.

\bibitem{T73}
N.S. Trudinger,
{\it Linear elliptic operators with measurable coefficients},
Ann. Scuola Norm. Sup. Pisa Cl. Sci. (3) 27 (1973), 265--308.

\bibitem{Tma}
N.S. Trudinger,
{\it Maximum principles for linear, non-uniformly elliptic operators with measurable coefficients},
Math. Z. 156(3) (1977), 291--301.

\bibitem{V92}
C. Vitanza,
{\it $W^{2,p}$-regularity for a class of elliptic second order equations with discontinuous coefficients},
Le Matematiche (Catania) 47(1) (1992), 177--186.
\end{thebibliography}
\end{document}